\documentclass{article}
    \usepackage{amsthm}
    \usepackage{amsmath}
    \usepackage{mathtools}
    \usepackage{float}
    \usepackage[english]{babel}
    \usepackage[autostyle]{csquotes}
    \usepackage[left=3cm, right=3cm, top=2cm, bottom=2cm]{geometry}
    \DeclareMathOperator{\sgn}{sgn}
    \newtheorem{definition}{Definition}
    \newtheorem{proposition}{Proposition}
    \newtheorem{theorem}{Theorem}
    \newtheorem{lemma}{Lemma}
    \newtheorem{remark}{Remark}
    \newtheorem{Example}{Example}
    \begin{document}
    \title{Besant Quadrilaterals}
    \author{Alan Horwitz}
    \date{5/6/26}
    \maketitle
    \begin{abstract}We solve the following problem of W.H. Besant using a formula for the coefficients of an ellipse inscribed in a quadrilateral, $Q$: \enquote{If an ellipse be inscribed in a quadrilateral so that one focus is equidistant from the four vertices(call that point $EP$), the other focus must be at the intersection of the diagonals(call that point $IP$).} We also prove somewhat more than just solving Besant's problem itself, though it would be nice to see the details of the geometric approach proposed by Besant. More precisely, we also prove the converse result and additional results when $Q$ is a trapezoid. Finally, we show that such an inscribed ellipse exists if and only if $Q$ is orthodiagonal.\end{abstract}
    \section{Introduction}
    This paper discusses the following problem proposed by W.H. Besant(it is listed as number 102 on page 86 of \cite{ref1}): \enquote{If an ellipse be inscribed in a quadrilateral so that one focus is equidistant from the four vertices, the other focus must be at the intersection of the diagonals.} Besant outlines a solution(see \cite{ref2}, page 39, solution number 104), but doesn't fill in the details. This author is grateful to Professor Mordechai Ben-Ari of the Weizmann Institute of Science in Israel for making me aware of this very interesting problem. Here we use a different approach based on this author's work(see, for example, \cite{ref5} and \cite{ref6}). We also prove somewhat more than just solving problem number 102 itself, though it would be nice to see the details of the geometric approach proposed by Besant. First we make the following definition.
    \begin{definition}A cyclic quadrilateral, $Q$, is called a \textbf{Besant quadrilateral} if there is an ellipse, $E_0$, inscribed in $Q$ such that one focus of $E_0$ is equidistant from the four vertices of $Q$. Such an ellipse, $E_0$, is called a \textbf{Besant ellipse} for $Q$. We also sometimes say that $E_0$ corresponds to(or is associated with) $Q$.\end{definition}
    Before proceeding, we introduce the following notation: For any convex quadrilateral, $Q$, let $IP=$ the intersection point of the diagonals of $Q$. If $Q$ is also a cyclic quadrilateral, let $EP=$the center of the circle passing though the vertices of $Q$.
    \begin{figure}[H]\includegraphics{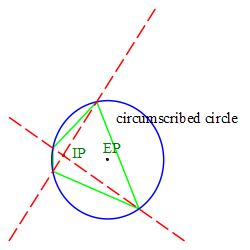}\end{figure}
    It is natural to ask whether there are any Besant quadrilaterals, $Q$. Indeed, even if Besant's problem holds true, it might be vacuously so. Since the definition above(and Besant's original problem) assumes that the inscribed ellipse has a focus at a point which is equidistant from the four vertices of $Q$, such a point must exist. But clearly if such a point exists, it must be the center of a circle passing though the vertices of $Q$--that is, the point must equal $EP$ and $Q$ must be cyclic. So throughout we restrict our attention to cyclic quadrilaterals. It turns out that there are cyclic quadrilaterals which are Besant, but not every cyclic quadrilateral is Besant. Indeed we show(Theorem \ref{Besant} below) that a cyclic quadrilateral, $Q$ is a Besant quadrilateral if and only if $Q$ is an orthodiagonal quadrilateral. Thus we have a characterization of cyclic, orthodiagonal quadrilaterals in terms of inscribed ellipses. It is also natural to ask about the converse: What if one assumes that one focus of the inscribed ellipse is at the intersection point of the diagonals ? Must the other focus be at the point equidistant from the four vertices ? We answer that in Theorem \ref{Besant}(i) below. Below is a Besant quadrilateral along with the corresponding Besant ellipse(see Example \ref{EX} below): \begin{figure}[H]\includegraphics{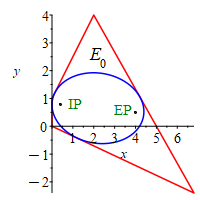}\label{Figure1}\end{figure}
    It is natural to ask what Besant's problem number 102 means if $EP$ equals $IP$, and also which cyclic quadrilaterals have this property. Now $E_0$ is a Besant ellipse for a cyclic quadrilateral, $Q$, if one focus of $E_0$ equals $EP$. But if $IP=EP$, then the \enquote{two} foci of $E_0$ coincide and thus $E_0$ is a circle. So Besant's problem makes sense when $IP=EP$ if we allow the Besant ellipse to be a circle inscribed in $Q$ and with center equal to $IP=EP$. However, we only consider the case when $IP \neq EP$. Towards that end, we have the following result, which we prove later in § Appendix along with some related results.
    \begin{proposition}\label{IPNEQEP} Suppose that $Q$ is a cyclic quadrilateral which is not a parallelogram. Then $IP \neq EP$.\end{proposition}
    What if $Q$ is a parallelogram ? It follows easily that $IP=EP$ if $Q$ is a rectangle, and a parallelogram is cyclic if and only if it is a rectangle. Thus if $Q$ is a cyclic quadrilateral which is a parallelogram, then $IP=EP$. Conversely, if $IP=EP$, then $Q$ is a parallelogram by Proposition \ref{IPNEQEP}. Thus throughout, we assume the following:
    \newline • a Besant quadrilateral is not a rectangle, and hence not a parallelogram.
    \newline • a Besant ellipse is not a circle.
    Another question to ask is about uniqueness, which we answer in the following result.
    \begin{proposition}\label{BEUN}If $Q$ is a Besant quadrilateral, then a Besant ellipse for $Q$ must be unique.\end{proposition}
    To prove this, we use the following result, which was stated by Chakerian in \cite{ref4}, but no proof is cited or given. We gave a full proof in \cite{ref5}.
    \begin{proposition}\label{P2}Suppose that $E_1$ and $E_2$ are distinct ellipses with the same center and which are each inscribed in a convex quadrilateral, $Q$. Then $Q$ must be a parallelogram.\end{proposition}
    (proof of Proposition \ref{BEUN}): Suppose that $Q$ is a Besant quadrilateral. Suppose that $E_1$ and $E_2$ are each Besant ellipses for $Q$. Then $E_1$ and $E_2$ have foci $EP$ and $IP$, which implies that $E_1$ and $E_2$ have the same center, $\frac{1}{2}(EP+IP)$. By Proposition \ref{P2}, $E_1$=$E_2$.\newline
    \section{Main Result}
    \textbf{Notation}: (i) For any quadrilateral, $Q$, we let $Q_{M}=$ quadrilateral with vertices at the midpoints of $Q$.\newline
    (ii) $Q(A_1,A_2,A_3,A_4)$ denotes the quadrilateral with vertices $A_1,A_2,A_3$, and $A_4$, starting with $A_1$ and going clockwise.\newline
    \begin{figure}[H]\includegraphics{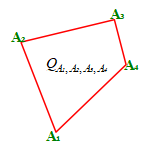}\end{figure}
    \begin{theorem}\label{Besant}Let $Q$ be a cyclic quadrilateral which is not a parallelogram.\newline
    (i) Let $E_0$ be an ellipse inscribed in $Q$ which is not a circle. If one focus of $E_0$ equals EP or IP, then the other focus of $E_0$ equals $IP$ or $EP$. In addition, $Q$ must be an orthodiagonal quadrilateral(thus $Q_{M}$ is cyclic) and the center of $E_0$ equals the center of the circle inscribed in $Q_{M}$.\newline
    (ii) Suppose that $Q$ is also an orthodiagonal quadrilateral. Then there exists an ellipse, $E_0$, inscribed in $Q$ whose foci are at $EP$ and at $IP$.\newline
    (iii) If $Q$ is also a trapezoid, then $Q$ is a Besant quadrilateral if and only if $Q$ is isoceles. In addition, the corresponding Besant ellipse is the ellipse of maximal area inscribed in $Q$.
    \end{theorem}
    Theorem \ref{Besant}(i) and (ii) can be phrased another way: A cyclic quadrilateral, $Q$, is a Besant quadrilateral if and only if $Q$ is an orthodiagonal quadrilateral. Thus we have a characterization of cyclic, orthodiagonal quadrilaterals in terms of inscribed ellipses.\newline
    \textbf{Preliminaries}\newline
    Before proving Theorem \ref{Besant}, we look at some preliminary material. To prove all of the results below for convex quadrilaterals in the plane, we work with the following quadrilateral of a special form. Let $Q_{s,t,v,w}$ denote the convex quadrilateral with vertices $(0,0),(0,1),(s,t)$, and $(v,w)$, where \begin{equation}\label{R1}
    s,t,v>0, t \geq w \end{equation}
    \begin{figure}[H]\includegraphics{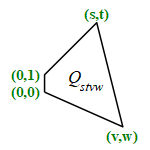}\end{figure}
    Let $Q$ be a convex quadrilateral. We say that $Q$ is a trapezoid if $Q$ has at least one pair of parallel sides. Assume throughout this section that all quadrilaterals are convex, but not trapezoids. We look at the trapezoid case later in \S \ref{Trapezoids}, where we prove some special properties of Besant trapezoids and the corresponding Besant ellipse.
    We find it useful to let \begin{equation}N=vt-ws\label{N}\end{equation}
    Since $Q_{s,t,v,w}$ is convex and not a trapezoid, it follows easily that \begin{equation}\label{R2}N+s-v>0,N>0,N \neq v,s \neq v\end{equation}
    The lines containing the sides of $Q_{s,t,v,w}$, going clockwise, are given by $S_1= \overleftrightarrow{(0,0) (0,1)}$,$S_2= \overleftrightarrow{(0,1) (s,t)}$,$S_3= \overleftrightarrow{(s,t) (v,w)}$, and $S_4= \overleftrightarrow{(0,0) (v,w)}$. The midpoints of the sides are $MP_1=(0,\frac{1}{2}) \in S_1$, $MP_2=(\frac{s}{2},\frac{1+t}{2}) \in S_2$, $MP_3=(\frac{s+v}{2},\frac{t+w}{2}) \in S_3$, and $MP_4=(\frac{v}{2},\frac{4}{2}) \in S_4$. The diagonals are $D_1 = \overline{(0,0) (s,t)}=$ diagonal from lower left to upper right and $D_2 = \overline{(0,1) (v,w)}=$diagonal from lower right to upper left, and the lines containing $D_1$ and $D_2$ have equations $y= \frac{t}{s}x$ and $y=1+\frac{w-1}{v}x$, respectively.
    The intersection point of the diagonals is given by \begin{equation}IP=\left(\frac{vs}{N+s},\frac{vt}{N+s}\right)\label{IP}\end{equation}
    Following are several polynomial expressions we use a lot below:
    \begin{gather}\beta=s^2+t^2\nonumber\\I=t^3-t^2+s^2(t+1)\nonumber\\H=sv \beta -tI\\\L=(v-s)\beta+2st\nonumber\label{HIL}\end{gather} and
    \begin{equation}CYC= \beta v-s(v^2+w^2)-N.\label{cyc}\end{equation}
    Note that $\sqrt \beta$= length of the diagonal $D_1$.
    First we state two results about when $Q_{s,t,v,w}$ is cyclic or orthodiagonal.\newline
    \begin{lemma}\label{cycorth}
    (i) $Q_{s,t,v,w}$ is cyclic $ \iff CYC=0$. Furthermore, the center of the circle passing thru the vertices of $Q_{s,t,v,w}$ is \begin{equation}EP=\left(\frac{\beta-t}{2s},\frac{1}{2}\right)\label{EP}\end{equation}\newline
    (ii) $Q_{s,t,v,w}$ is orthodiagonal $ \iff w=1-\frac{sv}{t}$.
    \end{lemma}
    \begin{proof} For (i), the equation of the circle passing thru $(0,0),(0,1),(s,t)$ is $f(x,y)=x^2+y^2+\left(\frac{t}{s}-\frac{t^2}{s}-s\right)x-y$. Now $Q_{s,t,v,w}$ is cyclic $\iff f(v,w)=-\frac{CYC}{s}=0 \iff CYC=0.$ For (ii), $Q_{s,t,v,w}$ is orthodiagonal $\iff$ the diagonals are perpendicular $\iff \frac{t}{s}\frac{w-1}{v}=-1 \iff w=1-\frac{sv}{t}$.\end{proof}
    The following lemma gives a useful necessary condition for $Q_{s,t,v,w}$ to be cyclic and orthodiagonal.
    \begin{lemma}[orth]\label{orth} If $Q_{s,t,v,w}$ is cyclic and orthodiagonal, then $H=0$.\end{lemma}
    \begin{proof}By Proposition \ref{cycorth}(ii) , $w=1-\frac{sv}{t}$, which implies that $CYC=-\frac{v}{t^2}H$. Since $v \neq 0,H=0$
    by Proposition \ref{cycorth}(i).\end{proof}
    The following useful fact is used below several times: \begin{equation}H=0 \iff v=\frac{tI}{s \beta}\label{v}\end{equation}
    The following lemma allows us to assume that Q=$Q_{s,t,v,w}$ in the proofs below.
    \begin{lemma}\label{scaling}Let T be the scaling transformation given by $T(x,y)=(kx,ky),k\neq 0$, let Q be a convex quadrilateral, and let $E_0$ be an ellipse. \newline(i) Then $T(E_0)$ is also an ellipse.\newline(ii) If $F_1$ and $F_2$ are the foci of $E_0$, then $T(F_1)$ and $T(F_2)$ are the foci of $T(E_0)$.\newline(iii) If Q is a cyclic quadrilateral, we let $C(Q)$ denote the center of the circle circumscribed about Q. Then $C\left(T(Q)\right)=T\left(C(Q)\right)$\end{lemma}
    \begin{remark}If $E_0$ is a circle with center $C_0$, then (ii) becomes $T(C_0)$ is the center of $T(E_0)$.\end{remark}
    We omit the details of the proof of Lemma \ref{scaling}.
    The following result links Q and $Q_{M}$. It is probably a known fact, but we provide the details of a proof here. It is not essential for our main results below, but it is also interesting in its own right.
    \begin{proposition}\label{orthomp}If Q is an orthodiagonal quadrilateral, then $Q_M$ is cyclic. Furthermore, if Q is also cyclic and $C_0$ and EP are the centers of the circles circumscribed about $Q_M$ and Q, respectively, then $C_0=\frac{1}{2}(EP+IP)$.
    \end{proposition}. We give the proof in the Appendix.
    For those reading the book \cite{ref6}, we have the following corrections and modifications: \begin{remark}In \cite{ref6} we referred to the lower left hand corner(LLC) vertex of Q. But we did not give a precise definition of LLC, and indeed there is no such definition in all cases. Also, we assumed that $t \> w$, but it is possible that $t=w$. For example, one could have $s=2,t=3,v=4,w=3$.\end{remark}
    The following result from \cite{ref6} is very useful for the rest of this paper(here these are the same coefficients, but written in a slightly different form).
    Let $J$ be the open line segment $(0,1)$. We introduce another expression used a lot below. \begin{equation}\tau =(s-v)r+v\end{equation}
    \begin{proposition}\label{ins}(i) $E_0$ is an ellipse inscribed in $Q_{s,t,v,w}$ if and only if the general equation of $E_0$ is given by $ \psi(x,y)=0$ for some $r \in J$, where \begin{equation}\psi(x,y)=A(r)x^2+B(r)xy+C(r)y^2+D(r)x+E(r)y+F(r)\label{psi}\end{equation} and \begin{gather}A(r)=\left((w-1)^2s^2-2v(wt+t-2w)s+t^2v^2\right)r^2+2v(st-2ws-tN)r+t^2v^2\nonumber\\B(r)=-2vs\left(2(v-s)r^2+(s-2v-N)r+vt\right)
    \\C(r)=s^2v^2,D(r)=2svr\left((N-s)r-N+sw\right)\nonumber\\E(r)=-2s^2v^2r,F(r)=s^2v^2r^2\nonumber\end{gather}\newline
    (ii) For any ellipse, $E_0$, inscribed in $Q_{s,t,v,w}$, the center of $E_0$ is $(x_0,y_0)$, where \begin{equation}x_0=\frac{sv}{2 \tau},y_0=\frac{(s-N)r+vt}{2 \tau}\end{equation}\label{center}\end{proposition}
    Key to our proof below of Lemma \ref{focip} is the following well-known result of Marden. We state the special case when the conic is an ellipse and use $t_k=\frac{m_k}{\sum_{i=1}^{3} m_k}$ from \cite{ref9}. For the results below, $z_1$,$z_2$, and $z_3$ are three distinct, noncollinear points in the complex plane and $L_1$,$L_2$, and $L_3$ are the line segments connecting $z_2$, and $z_3$, $z_1$ and $z_3$, and $z_1$ and $z_2$, respectively.\newline
    \textbf{Notation}: Throughout this section, given real numbers $t_1,t_2,t_3$ with $\sum_{i=1}^{3} t_k=1$, and given $t_1t_2t_3>0$ as above, we let \begin{equation}F(z)=\frac{t_1}{z-z_1}+\frac{t_2}{z-z_2}+\frac{t_3}{z-z_3}\label{F(z)}\end{equation}
    \begin{theorem}[Marden]\label{Marden}Let $Z_1$ and $Z_2$ denote the zeros of $F(z)$. If $t_1t_2t_3>0$ then $Z_1$ and $Z_2$are the foci of an ellipse, $E_0$, which is tangent to $L_1,L_2$, and $L_3$ at the points $\zeta_1, \zeta_2,\zeta_3$, where \begin{gather}\zeta_1=\frac{t_2z_3+t_3z_2}{t_2+t_3}\nonumber\\\zeta_2=\frac{t_1z_3+t_3z_1}{t_1+t_3}\\\zeta_3=\frac{t_1z_2+t_2z_1}{t_1+t_2},
    \nonumber\end{gather}\label{zeta123}respectively.
    \end{theorem}
    \begin{remark}Though it is not stated explicitly, in the case when $E_0$ is a circle, the foci $Z_1$ and $Z_2$ are identical and correspond to the center of that circle. Also, $F(z)$ has a double root at $Z_1$=$Z_2$.\end{remark}
    We also need the following result--sort of a converse of Marden's Theorem. We really only need the result for non-circular ellipses, but we state it for circles as well.
    \begin{theorem}\label{MTC}Suppose that $E_0$ is an ellipse tangent to $L_1$,$L_2$, and $L_3$.\newline
    (i) If $E_0$ is not a circle and $Z_1$ and $Z_2$ are the foci of $E_0$, then there exists $t_1,t_2,t_3$ with $t_1t_2t_3>0$ and $\sum_{i=1}^{3} t_k=1$ such that $F(z)$ has zeros $Z_1$ and $Z_2$, where $F$ is given by (\ref{F(z)}).\newline
    (ii) If $E_0$ is a circle and $Z_1$ is the center of $E_0$, then there exists $t_1,t_2,t_3$ with $t_1t_2t_3>0$ and $\sum_{i=1}^{3} t_k=1$ such that $F(z)$ has a double zero at $Z_1$, where $F$ is given by (\ref{F(z)}).\end{theorem}
    Before proving Theorem \ref{MTC}, we need the following results--the first two are from \cite{ref7}.
    \begin{proposition}\label{triins}Let $T$ be the triangle with vertices $(0,0),(1,0)$, and $(0,1)$. Then\newline
    (i) $E_0$ is an ellipse inscribed in $T$ if and only if the general equation of $E_0$ is given by \begin{equation}w^2x^2+t^2y^2-2wt(2wt-2w-2t+1)xy-2w^2tx-2t^2wy+t^2w^2=0\label{wt}\end{equation} for some $(w,t) \in S=(0,1)\times(0,1)$. Furthermore,\newline(ii) If $E_0$ is the ellipse given in (i) with equation \ref{wt} for some $(w,t) \in S$, then $E_0$ is tangent to the three sides of $T$ at the points $T_1=(t,0),T_2=(0,w)$, and $T_3=\left(\frac{t(1-w)}{t+(1-2t)w},\frac{w(1-t)}{t+(1-2t)w}\right)$.\end{proposition}
    \begin{theorem}\label{Tbound}Let $P_1$ and $P_2$ be distinct points which lie on different sides of the boundary, $\partial(T)$, of a triangle, $T$, and assume that neither $P_1$ nor $P_2$ equals one of the vertices of $T$. Then there is a unique ellipse inscribed in $T$ which is tangent to $\partial(T)$ at $P_1$ and at $P_2$.\end{theorem}
    \begin{lemma}\label{tj}Suppose that $E_0$ is an ellipse which is tangent to $L_1$,$L_2$, and $L_3$ at $\zeta_1, \zeta_2$, and $\zeta_3$, respectively. Then there exists $t_1,t_2,t_3$ with $t_1t_2t_3>0$ and $\sum_{i=1}^{3} t_k=1$ such that $\zeta_1, \zeta_2$, and $\zeta_3$ are given by (\ref{zeta123}).\end{lemma}
    \begin{proof}(of Lemma \ref{tj}): By affine invariance, it suffices to prove Lemma\ref{tj} when $z_1=0,z_2=i$, and $z_3=1$, that is when $L_1$,$L_2$, and $L_3$ enclose the unit triangle. Then (\ref{zeta123}) becomes $\zeta_1=\frac{t_2+t_3i}{t_2+t_3}, \zeta_2=\frac{t_1}{t_1+t_3},\zeta_3=\frac{t_1i}{t_1+t_2}$. Since $E_0$ is tangent to $L_1$,$L_2$, and $L_3$ at $\zeta_1, \zeta_2$, and $\zeta_3$, we also have $\zeta_1=\lambda_1i+1-\lambda_1,\zeta_2=1-\lambda_2$, and $\zeta_3=(1-\lambda_3)i$ for some $0<\lambda_1,\lambda_2,\lambda_3<1$. Ignoring redundancies, $\zeta_1, \zeta_2$, and $\zeta_3$ are given by (\ref{zeta123}) if and only if \begin{gather}\frac{t_3}{t_2+t_3}=\lambda_1\nonumber\\\frac{t_1}{t_1+t_3}=1-\lambda_2\\\frac{t_1}{t_1+t_2}=1-\lambda_3\nonumber\label{tlm}\end{gather}
    By Proposition \ref{triins}, $\lambda_2=1-t,\lambda_3=1-w$, and $\lambda_1=\frac{w(1-t)}{t+(1-2t)w}$ for some $(w,t) \in S$, , which implies that $\lambda_1=\frac{(1-\lambda_3)\lambda_2}{\lambda_2+\lambda_3-2\lambda_2\lambda_3}$. Substituting for $\lambda_1$ and using $t_3=1-t_1-t_2$ yields the unique solution for (\ref{tlm}) given by $t_1=\frac{1+\lambda_2\lambda_3-\lambda_2-\lambda_3}{1-\lambda_2\lambda_3}=\frac{tw}{t+(1-t)w}, t_2=\lambda_3\frac{1-\lambda_2}{1-\lambda_2\lambda_3}=\frac{t(1-w)}{t+(1-t)w}$. Now $t+(1-t)w>0$ and hence $t_1,t_2>0$. Also, $t_3=\frac{(1-t)w}{t+(1-t)w}>0$, which finishes the proof of Lemma \ref{tj}.\end{proof}
    \begin{proof}(of Theorem \ref{MTC}): Suppose first that $E_0$ is not a circle(thus $Z_1 \neq Z_2$) and that $E_0$ is tangent to $L_1,L_2$, and $L_3$ at $\zeta_1, \zeta_2,\zeta_3$, respectively. By Lemma \ref{tj}, there exists $t_1,t_2,t_3$ with $t_1t_2t_3>0$ and $\sum_{i=1}^{3} t_k=1$ such that $\zeta_1, \zeta_2$, and $\zeta_3$ are given by \ref{zeta123}. Let $W_1$ and $W_2$ be the zeros of $F(z)$, where $F$ is given by (\ref{F(z)}). By Theorem \ref{Marden}, $W_1$ and $W_2$ are the foci of an ellipse, $E_1$, which is tangent to $L_1,L_2$, and $L_3$ at $\zeta_1, \zeta_2,\zeta_3$, respectively, where again, $\zeta_1, \zeta_2$, and $\zeta_3$ are given by \ref{zeta123}. Thus $E_0$ and $E_1$ are tangent to $L_1,L_2$, and $L_3$ at the same points. By Theorem \ref{Tbound}, $E_0=E_1$ and thus $Z_1$ and $Z_2$ and $W_1$ and $W_2$ are identical pairs, which implies that $F(z)$ has zeros $Z_1$ and $Z_2$. The case when $E_0$ is a circle is proved similarly, where in that case $Z_1=Z_2$ and $F(z)$ has a double zero at $W_1=W_2$.\end{proof}
    \textbf{Note}: While Lemma \ref{tj} is affine invariant, Theorem \ref{MTC} is not since it involves the foci of an ellipse.\newline
    Again, we really only need the following result for non-circular ellipses, but we state it for circles as well. Recall that $\tau=(s-v)r+v$.
    \begin{lemma}\label{focip}
    Suppose that $E_0$ is an ellipse inscribed in $Q_{s,t,v,w}$.\newline
    (i) If $E_0$ is not a circle, then the foci, $Z_1$ and $Z_2$, of $E_0$ are given by the roots of the quadratic polynomial \begin{gather}p(z)=z^2+\frac{-sv+(N-s)r-vt)i}{\tau}z\nonumber\\\frac{rs(-w+vi)}{\tau}\label{p}\end{gather}.\newline
    (ii) If $E_0$ is a circle with center $Z_1$, then $p$ has a double root at $z=Z_1$.\end{lemma}
    \begin{proof}Recall that $t_3=1-t_1-t_2$. The sides of $Q_{s,t,v,w}, S_1,S_3$, and $S_4$, form a triangle, $\triangle$, whose vertices are the complex points
    \begin{equation}z_1= \left(t-s\frac{t-w}{s-v} \right)i, z_2=0, z_3=v+wi\label{T1}\end{equation}
    and $F(z)=-\frac{i(v-s)p(z)}{(z(i(v-s)z+N)(-z+v+iw)}$, where $F$ is given by (\ref{F(z)}) and \begin{gather}p(z)=z^2-p_1z-p_0\\p_1=(t_1+t_2)v-
    \frac{v(t-w)t_1+w(s-v)t_2-N}{v-s}i,\nonumber\\p_0=\frac{N(v+iw)t_2}{i(v-s)}\nonumber. \label{t1t2}\end{gather}.
    Clearly $E_0$ is tangent to the sides of the triangle, $\triangle$, since $E_0$ is inscribed in $Q_{s,t,v,w}$.
    To prove (i), by Theorem \ref{MTC}(i), $Z_1$ and $Z_2$ are the zeros of $F(z)$ for some $t_1,t_2,t_3$ with $t_1t_2t_3>0$ and $\sum_{i=1}^{3} t_k=1$, and such that $z_1, z_2$, and $z_3$ are given by (\ref{T1}). Thus we may write $p(z)=(z-Z_1)(z-Z_2)=z^2-(Z_1+Z_2)z+Z_1Z_2$, and thus by (\ref{t1t2}) $Z_1+Z_2=p_1$. Since $Z_1$ and $Z_2$ are the foci of $E_0$ it also follows that $\frac{1}{2}(Z_1+Z_2)=(x_0,y_0)=$ the center of $E_0$. By Proposition \ref{ins}(ii) and (\ref{center}), $(t_1+t_2)v=\frac{sv}{\tau}$ and $-\frac{v(t-w)t_1+w(s-v)t_2-N}{v-s}=\frac{(s-N)r+vt}{\tau}$. Solving this system of equations for $t_1,t_2$ easily yields $t_1=s\frac{(s-v)r+N}{(N\tau}$ and $t_2=\frac{s(v-s)r}{(N\tau}$. Substituting for $t_1,t_2$ in (\ref{t1t2}) then yields (\ref{p}). To prove (ii), by Theorem \ref{MTC}(ii), $p$ has a double zero at $Z_1$. $p'(z)=2z-p_1 \Rightarrow p'(Z_1)=2Z-1-p_1=0 \Rightarrow p_1=2Z-1$. Arguing as in the non-circle case above again yields (\ref{p}). \end{proof}
    \begin{lemma}\label{wv}Suppose that $v=\frac{tI}{s \beta},w= 1-\frac{sv}{t}$, and $r=\frac{v}{s+v}$. Then $p\left(\frac{\beta-t}{2s}+\frac{1}{2}i\right)=p\left(\frac{v}{N+s}(s+it)\right)=0$, where $p$ is given by (\ref{p}).\end{lemma}
    \begin{proof}This is a straightforward algebraic calculation and we omit the details.\end{proof}
    \textbf{Proof of Theorem Besant}\newline We prove (i) and (ii), with (iii) proven in \S \ref{Trapezoids}.
    \begin{proof}As noted in the proof of Proposition \ref{orthomp}, it suffices to assume that $Q=Q_{s,t,v,w}$ for some $s,t,v,w$ satisfying (\ref{R1}) and (\ref{R2}). To prove (i), suppose first that one focus of $E_0$ equals
    $EP=\frac{\beta-t}{2s}+\frac{1}{2}i$, which implies that $p(EP)=0$ by Lemma \ref{focip}, where $p$ is given by (\ref{p}). Define the following expressions in $s,t,v,w$: \begin{gather}a_{O_1} = s^3-vs^2+t(t-2w)s+t(t-2)v\nonumber\\a_{O_2} = (t+2w-1)s^3-v(1+3t)s^2+ t^2(t-1)(s-v).\nonumber\end{gather}Now \begin{gather}p(EP)=\frac{(s+(t-1)i)(sO_1-iO_2)}{4s^2 \tau}\\O_1=a_{O_1}r-v(\beta-2t), O_2=a_{O_2}r+vI.\nonumber\end{gather}\label{pP1}Since $s \neq 0$ by (\ref{R1}), $p(EP)=0 \iff sO_1-iO_2=0$, which then implies $O_1=O_2=0$. That, in turn, gives $a_{O_1}a_{O_2}r-v(\beta-2t)a_{O_2}=0$ and $a_{O_1}a_{O_2}r+vIa_{O_1}=0$. Subtracting and dividing thru by $v$ yields $a_{O_1}I+(\beta-2t)a_{O_2}=0$, which holds $\iff a_5v+a_6=0$, where \begin{gather}a_5=s(I+ t \beta-2t^2)\\a_6=(I+ (t^2-s^2)(\beta-2t+1))w-t \beta(\beta-2t+1)\nonumber\label{a5a6}\end{gather}
    Assume first that $a_5 \neq 0$: Then $v=-\frac{a_6}{a_5}$, and substituting for $v$ in (\ref{cyc}) yields \begin{equation}
    CYC=\frac{\beta(\beta-2t+1)(I-\beta w)(t(\beta-2t)+\beta w)}{s(t(\beta-2t)+I)^2}.\end{equation} Since $CYC=0$, we have $I-\beta w=0$, which implies that $w=\frac{I}{\beta}$, or $t(\beta-2t)+\beta w=0$, which implies that $w=-\frac{t(\beta-2t)}{\beta}$. The first value of $w$ gives $a_6= -\frac{s^2(\beta-2t)(t(\beta-2t)+I)}{\beta}$, which implies that $v=\frac{s(\beta-2t)}{\beta}$, which in turn yields $vt-ws=\frac{s(\beta-2t)-I)}{\beta} = -s<0$.
    By (\ref{R1}) we must have the second value of $w$, which gives $a_6= -t\frac{I(t(\beta-2t)+I))}{\beta}$ and it follows that $v$ is given by (\ref{v}). It then follows immediately that $1-\frac{sv}{t}=-\frac{t(\beta-2t)}{\beta}$ and thus $Q_{s,t,v,w}$ is orthodiagonal by Lemma \ref{cycorth}(ii). Substituting for $v$ and for $w$ then yields $a_{O_1}=(\beta-2t)\frac{s^2(\beta)+tI}{s \beta}$, which implies that $O_1=\frac{(\beta-2t)}{s\beta}(s^2 \beta r+tI(r-1)$. Setting $O_1=0$ and solving for $r$ gives $r=\frac{tI}{s^2 \beta+tI}$. Using (\ref{v}) and simplifying yields $r=\frac{v}{s+v} \in J=(0,1)$. One obtains the same result for $r$ by setting $O_2=0$, though setting $O_1=0$ is sufficient. Since $IP=\frac{v}{N+s}(s+it)$, by Lemma \ref{wv} we have $p(IP)=0$.\newline
    Assume now that $a_5=0$: Since $s \neq 0$ by (\ref{R1}),
    (\ref{a5a6}) implies that $I+ t \beta-2t^2=0$, and thus \begin{gather}s^2=Y,\nonumber\\Y=\frac{3t^2-2t^3}{2t+1}\nonumber.\end{gather}Since $s>0$, we have \begin{equation}s= \sqrt Y.\label{st}\end{equation} Substituting for $s^2$ then yields the following: \begin{gather}t^2-s^2 =2t^2\frac{(2t-1}{1+2t},I=\frac{2t^2}{1+2t}\\\beta=\frac{4t²}{1+2t},\beta-2t+1=\frac{1}{1+2t}\nonumber.\label{ssq}\end{gather} By (\ref{ssq}) we have $a_6=\left(\frac{2t^2}{1+2t}+2t^2\frac{(2t-1}{1+2t}\frac{1}{1+2t}\right)w-t\frac{4t^2}{1+2t}=\frac{4t^3(2w-1)}{(1+2t)^2}$. Now $a_5v+a_6=0 \Rightarrow a_6=0$, which implies that $w=\frac{1}{2}$. Substituting for $s$ using (\ref{st}) and $w=\frac{1}{2}$ in (\ref{cyc}) yields \begin{equation}CYC=vt\frac{2t-1}{2t+1}-\frac{4v^2-1}{4} \sqrt Y\label{CYC}.\end{equation}Note that if $v=\frac{1}{2}$, then using (\ref{CYC}) and $CYC=0$ gives $t=\frac{1}{2}$, which in turn implies that $s=\frac{1}{2}$ by (\ref{st}), which contradicts the assumption that $v \neq s$. Thus we may assume that $v \neq \frac{1}{2}$. Then \begin{equation}CYC=0 \Rightarrow \sqrt Y=\frac{4vt(2t-1)}{(2t+1)(4v^2-1)}\label{2}.\end{equation} Squaring both sides, simplifying, and factoring gives $t^2p_1(t,v)p_2(t,v)=0$, where $p_1(t,v)=4(2t+1)v^2+2t-3$ and $p_2(t,v)=4(2t-3)v^2+2t+1$.\newline
    case 1: $p_1(t,v)=0$. Then $t=q_1(v)=\frac{1}{2}\frac{3-4v^2}{4v^2+1}$. Substituting for $t$ gives $Y=4v^2q_1^2(v) \Rightarrow
    Y^\frac{1}{2}=2vq_1(v)$ since $t>0 \Rightarrow 3-4v^2>0$. $t=q_1(v)$ also implies that $\frac{4vt(2t-1)}{(2t+1)(4v^2-1)}=-2vq_1(v)$ , which implies that (\ref{2}) cannot hold unless $q_1(v)=0 \iff v=\frac{\sqrt3}{2}$. But then $t=0$. Thus case 1 cannot hold.\newline
    case 2: $p_2(t,v)=0$. Then \begin{equation}t=q_2(v)=\frac{1}{2}\frac{12v^2-1}{4v^2+1}\label{t}\end{equation} and $Y=\frac{1}{4v^2}q_2^2(v) \Rightarrow Y^\frac{1}{2}=\frac{1}{2v}q_2(v)$ since $t>0 \iff 12v^2-1>0$. It then follows easily from (\ref{CYC}) that $CYC=0$.
    Also, \begin{equation}s=\frac{1}{2v}q_2(v)\label{s}\end{equation} and using both (\ref{s}) and (\ref{t}) gives $1-\frac{sv}{t}= \frac{1}{2}=w$ and thus $Q_{s,t,v,w}$ is orthodiagonal by Lemma \ref{cycorth}(ii). That also gives $a_{O_1}=-\frac{16v^4+16v^2-1}{32v^3}q_2(v)$ and setting $O_1=a_{O_1}r-v (\beta-2t) =0$ gives $r=\frac{4v^2(1+4v^2)}{16v^4+16v^2-1}=\frac{v}{s+v} \in J=(0,1)$.
    Now by (\ref{ssq}), $\frac{tI}{s \beta}=\frac{t}{s}\frac{2t^2}{1+2t}\frac{1+2t}{4t^2}=\frac{t}{2s}=v$, and again by Lemma \ref{wv} we have $p(IP) = 0$. Thus if $a_5 \neq 0$ or $a_5=0$, $IP$ is the other focus of $E_0$, which proves that if one focus of $E_0$ equals $EP$, then the other focus of $E_0$ equals $IP$.
    Conversely, suppose that one focus of $E_0$ equals $IP$, which implies that $p(IP)=0$. Define the following expressions in $s,t,v,w$: \begin{gather}
    a_{O_3} = s(sw-2vt)(w-1)+ t(t-1)v^2\nonumber\\a_{O_4} = (v^2-w(w-1^2)s^3+ s^2v( t(w-1)^2-v^2)+t^2v^2(w-1)s-t^2(t-1)v^3.\nonumber\end{gather}
    Now $p(IP)=p\left(\frac{v}{N+s}(s+it)\right)=\frac{O_3+2svO_4i}{(N+s)^2 \tau}$, where $O_3=a_{O_3}r+ vt(v-s-N), O_4=a_{O_4}r+v^2(s^2-t^2)(v-s-N)$. $p(IP)=0 \iff O_3=O_4=0$. $O_3=0 \Rightarrow a_{O_3}a_{O_4}r+ vt(v-s-N)a_{O_4}=0$ and $O_4=0 \Rightarrow a _{O_3}a_{O_4}r+v^2(s^2-t^2)(v-s-N)a_{O_3}=0$. Subtracting gives $vt(v-s-N)a_{O_4}-v^2(s^2-t^2)(v-s-N)a_{O_3}=0$, and after simplifying, one has $svN(N+s)(N+s-v)(sv+tw-t)=0$. By \ref{R1} and \ref{R2}, the factors $s,v,N,N+s$, and $N+s-v$ are each positive and thus non-zero. Thus $sv+tw-t=0$, which implies that $w=1-\frac{sv}{t}$ and again $Q_{s,t,v,w}$ is orthodiagonal by Lemma \ref{cycorth}(ii). By Lemma \ref{orth}, $H=0$, which implies that {v} is given by (\ref{v}). Substituting for $v$ and for $w$ then yields $a_{O_3}=\frac{t(\beta-t)(s+v)I}{s \beta}$. As noted above, $v-s-N \neq 0$, which implies that $a_{O_3}\neq 0$ as well since $O_3=0$. Hence, setting $O_3=0$ implies that $r=-\frac{( vt(v-s-N)}{a_{O_3}}=
    -\frac{(stv \beta (v-s-N)}{t(\beta-t)(s+v)I}$, which simplifies to $\frac{v}{s+v}\in J $ using (\ref{v}). One obtains the same result for $r$ by setting $O_4=0$, though setting $O_3=0$ is sufficient. Since $EP=\frac{\beta -t}{2s}+\frac{1}{2}i$, by Lemma \ref{wv} we have $p(EP) = 0$. That proves that if one focus of $E_0$ equals $IP$, then the other focus of $E_0$ equals $EP$. Since the foci of $E_0$ are at $EP$ and at $IP$, the center of $E_0$ equals $\frac{1}{2}(EP+IP)=C_0=$ center of the circle inscribed in $Q_{M}$ by Proposition \ref{orthomp}. That proves Theorem \ref{Besant}(i). To prove (ii), since $Q_{s,t,v,w}$ is orthodiagonal, $w=1-\frac{sv}{t}$. As done above, $H=0$ by Lemma \ref{orth}, which implies that {v} is given by (\ref{v}). Let $E_0$ be the ellipse with equation \ref{psi} and with $r=\frac{v}{s+v}$. Then (ii) follows immediately by Lemma \ref{wv}. \end{proof}
    \begin{remark}In the beginning of the proof above, one could also look at the triangle formed by $S_2,S_3$, and $S_4$, which has vertices
    $z_1=\frac{s(v+wi)}{sw-(t-1)v},z_2=s+ti$, and $z_3=v+wi$. This leads to the same polynomial as $p(z)$ above. In fact, we do that below for the case when $Q$ is a trapezoid.\end{remark}
    \begin{Example}\label{EX}Let $s=2, t=4,v=\frac{34}{5}$, and $w=-\frac{12}{5}$: Then $EP= \left(4,\frac{1}{2}\right)$ and $IP=\left(\frac{2}{5},\frac{4}{5}\right)$. Letting $r=\frac{17}{22}$ in Proposition \ref{ins} yields the inscribed ellipse, $E_0$, with equation $649x^2+216xy+1936y^2-2996x-2992y=- 1156$. The foci of $E_0$ are $EP$ and $IP$ and thus $E_0$ is a Besant ellipse inscribed in $Q_{s,t,v,w}$, which is thus a Besant quadrilateral. See \ref{Figure1} above.\end{Example}
    \section{Trapezoids}\label{Trapezoids}\begin{proof}(of Theorem \ref{Besant}(iii)). Earlier, when working with non-trapezoids, it sufficed to consider the special quadrilateral $Q_{s,t,v,w}$, the convex quadrilateral with vertices $(0,0),(0,1),(s,t)$, and $(v,w)$. Again, by using an appropriate isometry, if $Q$ is a trapezoid, it suffices to consider $Q_{s,t,v,w}$ with $v=s$, which we denote by $Q_{s,t,s,w}$. Note that if $v=s$ and $t=w$, then $(s,t)=(v,w)$. So if $v=s$, then $t \neq w$, and thus (\ref{R1}) is equivalent to $s>0,t>w$. It is also easy to show that $Q_{s,t,s,w}$ is a parallelogram $\iff t-w-1=0$. Since we assumed that $Q$ is not a parallelogram(and hence $Q_{s,t,s,w}$ as well), it follows that $t-w-1 \neq 0$. Now for (\ref{R2}), we no longer assume $s \neq v$, of course. $N+s-v>0$ and $N>0$ are identical and equivalent to $t>w$, which has already been assumed. $N \neq v$ is equivalent to $t-w-1 \neq 0$, which has already been assumed above. Finally, Proposition \ref{cycorth} holds even if $Q_{s,t,v,w}$ is a trapezoid. In that case, substituting $v=s$ gives $CYC=s(t+w-1)(t-w)$. Since $t \neq w$ and $s \neq 0$, $CYC=0 \iff w=1-t$. That yields the cyclic trapezoid $Q_{s,t,s,1-t}$. Also, $t>w$ becomes $t>frac{1}{2}$ and $t-w-1 \neq 0$ becomes $t \neq 1$. Summarizing,
    If $Q_{s,t,s,w}$ is a not a parallelogram, then $Q_{s,t,s,w}$ is cyclic if and only if \begin{equation}w=1-t, t>\frac{1}{2},t \neq 1\label{cyctrap}\end{equation}
    Now suppose that $E_0$ is an ellipse inscribed in $Q_{s,t,s,1-t}$ and that $E_0$ is not a circle. One can show directly, or by letting $v=s$ and $w=1-t$ in Proposition \ref{ins}(i), that the equation of $E_0$ is given by $\psi(x,y)=0$, where $\psi(x,y)=A(r)x^2+B(r)xy+C(r)y^2+D(r)x+E(r)y+F(r)$, and $A(r)=4(t-1)^2r^2-4(t-1)^2r+t^2, B(r)=2st(2r-1), C(r)=s^2, D(r)=2rs( 2(t-1)r-3t+2),E(r)=-2s^2r,F(r)=s^2r^2, r \in J=(0,1)$. The sides of $Q_{s,t,s,1-t}, S_2,S_3$, and $S_4$, form a triangle whose vertices are the complex points $z_1=\frac{s}{2(1-t)}+\frac{1}{2}i, z_2=s+ti$, and $z_3=s+(1-t)i$. As in the proof for the non-trapezoid case, let $F(z)=\frac{t_1}{z-z_1}+\frac{t_2}{z-z_2}+\frac{t_3}{z-z_3},\sum_{i=1}^{3} t_k=1$, and let $Z_1$ and $Z_2$ denote the zeros of $F(z)$. Let $p_{T}(z)=(z-Z_1)(z-Z_2)$. Proceeding as in the proof of Lemma \ref{focip}(we omit all of the details), $Z_1$ and $Z_2$ are the foci of $E_0$ and one can show that \begin{gather}\frac{p_{T}}{2(t-1)}=p(z), \label{PT}\\p(z)=z^2+(-s+(2rt-2r-t) i)z+\nonumber\\i(s+i(1-t))r,r \in J=(0,1)\nonumber.\end{gather}
    Alternatively, one can use a limiting argument as $v \rightarrow s$. Then letting $v=s$ and $w=1-t$ in (\ref{p}) yields (\ref{PT}) above. For $Q_{s,t,s,1-t}$, the vertices are $(0,0),(0,1),(s,t)$, and $(s, 1-t)$, which implies that $IP= \left(\frac{s}{2t},\frac{1}{2}\right)$ and $EP= \left(\frac{s^2+t^2-t}{2s},\frac{1}{2}\right)$. Using complex notation, $p(EP)=p\left(\frac{s^2+t^2-t}{2s}+\frac{1}{2}i\right)=\frac{t^2-s^2+2st(2r-1)i}{4s^2}((t-1)^2+s^2)$. Thus $p(EP)=0 \iff t^2-s^2t=0$ and $2st(2r-1)i=0 \iff$ \begin{equation}s=t \hspace{1cm} \text{and} \hspace{1cm} r=\frac{1}{2}.\label{str}\end{equation}
    $p(IP)=p\left(\frac{1}{2}+\frac{1}{2}i\right)=\frac{1}{2}i((s+t-1)(2r-1)+(s-t)i)=0 \iff$ (\ref{str}) holds. Clearly, then, $p(EP)=0 \iff p(IP)=0 \iff$ \ref{str} holds. $Q_{s,t,s,1-t}=Q_{t,t,t,1-t}$ now has vertices $(0,0),(0,1),(t,t)$, and $(t, 1-t)$, which implies that the diagonals are perpendicular and hence $Q_{t,t,t,1-t}$ is orthodiagonal. Arguing as in the non-trapezoid case, the center of $E_0$ equals the center of the circle inscribed in $Q_{M}$. That proves Theorem \ref{Besant}(i) when $Q$ is a trapezoid. To prove Theorem \ref{Besant}(ii), if $Q_{s,t,s,w}$ is an orthodiagonal quadrilateral, then by Proposition \ref{cycorth}(ii), $w=1-\frac{sv}{t}=1-\frac{s^2}{t}$. Also, by (\ref{cyctrap}) $w=1-t$, which implies that $1-\frac{s^2}{t}=1-t \Rightarrow (t-s)\frac{(t+s}{t}=0 \Rightarrow s=t$. Arguing as above, it follows easily that if $r=\frac{1}{2}$, then $E_0$ is an ellipse inscribed in $Q_{t,t,t,1-t}$ whose foci are at $EP$ and at $IP$. Finally, to prove Theorem \ref{Besant}(iii), first, $Q_{t,t,t,1-t}$ is clearly an isoceles trapezoid. Second, let $a$ and $b$ denote the lengths of the semi-major and semi-minor axes, respectively, of $E_0$, where $E_0$ is an ellipse inscribed in $Q_{t,t,t,1-t}$. By (\cite{ref5}, Lemma A.2), $a^2b^2= -\frac{1}{4}t^2(2t-1)( r^2-r)$, which attains its maximum when $r=\frac{1}{2}$. Since $Q_{s,t,s,1-t}$ is a Besant quadrilateral if and only if (\ref{str}) holds, the corresponding Besant ellipse is the ellipse of maximal area inscribed in $Q_{t,t,t,1-t}$.\end{proof}
    \textbf{Example}: Consider the isoceles trapezoid $Q_{t,t,t,1-t}$ with $t=2$ and the ellipse, $E_0$, inscribed in $Q_{2,2,2,-1}$. $E_0$ has equation $3x^2+4y^2-6x-4y= -1$ and the foci of $E_0$ are $EP=\left(\frac{3}{2},\frac{1}{2}\right)$ and $IP=\left(\frac{1}{2},\frac{1}{2}\right)$. Thus $E_0$ is a Besant ellipse inscribed in $Q_{2,2,2,-1}$.
    \begin{figure}[H]\includegraphics{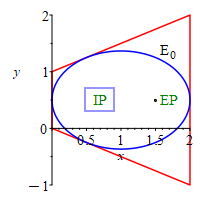}\end{figure}
    \begin{remark}If $Q$ is a Besant quadrilateral, but not a trapezoid, then a Besant ellipse for $Q$ is not necessarily the ellipse of maximal area inscribed in $Q$. See example\ref{EX} done above: $r=\frac{17}{22}$ yields the Besant ellipse for $Q_{2,4,34/5,-12/5}$, while $r=\frac{1}{123}(-151+ \sqrt64621)$ yields the ellipse of maximal area inscribed in $Q_{2,4,34/5,-12/5}$.\end{remark}
    \section{Alternate Approach}
    Here we give a somewhat different approach to the proof of the first part of Theorem \ref{Besant}(i). The other parts of the proof of Theorem \ref{Besant} would follow similarly. We use the formula below for the foci of an ellipse given its coefficients(we leave out the details of the derivation). Recall that $\tau=(s-v)r+v$.
    \begin{theorem}\label{foci}Let $E_0$ be an ellipse which is not a circle and with equation $\phi(x,y)=0$, where $\phi(x,y)=Ax^2+Bxy+Cy^2+Dx+Ey+F, A,C>0.$ Let $F_1$ and $F_2$ denote the foci of $E_0$, with $F_2=(x_{c},y_{c})$ the rightmost focus(if the rotation angle equals $\frac{\pi}{2}$, we let $F_2$ denote the uppermost focus) and let $(x_0,y_0)$ be the center of $E_0$. Let $\Delta=4AC-B^2,\delta=CD^2+AE^2-BDE-F\Delta, \mu=\frac{4\delta}{\Delta^2},M=(A-C)^2+B^2, k_{A}=\frac{\mu}{2}(C-A+\sqrt M)$, and $k_{C}=\frac{\mu}{2}(A-C\sqrt M)$.\newline
    (i) If $B \neq 0$, then the foci of $E_0$ are given by \begin{gather}F_1=(x_0-\sqrt k_{A},y_0+(sgnB)\sqrt k_{C})\nonumber\\F_2=(x_0+\sqrt k_{A},y_0-(sgnB)\sqrt k_{C})\nonumber\end{gather}
    (ii) If $B=0$, then the foci are $(x_0,y_0 \pm \sqrt k_{C})$ if $A>C$, and $(x_0 \pm \sqrt k_{A},y_0)$ if $A<C$.
    \end{theorem}
    \begin{lemma}[fociprod]\label{fociprod}Let $Z_1$ and $Z_2$ denote the foci of an ellipse, $E_0$, inscribed in $Q_{s,t,v,w}$, where $Z_1$ and $Z_2$ are written as complex numbers. Then $Z_1Z_2=\frac{rs(-w+vi)}{\tau}$.
    \end{lemma}
    \begin{proof}: We prove the case when $B \neq 0$. Then $\sqrt {M} > \lvert(A-C)\rvert \Rightarrow \sqrt {M}+C-A>0$ and $\sqrt {M}+A-C>0$. Since $E_0$ is an ellipse, it follows easily that $ \Delta >0$ and $ \delta >0$(see, for example, \cite{ref10} or \cite{ref11}). Thus $ \mu >0$, which implies that $k_{A}$ and $k_{C}$ from Theorem \ref{foci} are both positive. Now $k_{A}k_{C}=\frac{\mu^2}{4}(M-(A-C)^2)=\frac{\mu^2}{4}B^2$, which implies that $\sqrt (k_{A}k_{C})=\frac{1}{2}( \mu \lvert B \rvert)$. Now by Theorem \ref{foci}, $Z_1Z_2=x^2-y^2+k_{C}-k_{A}+2(x_0y_0+\sqrt (k_{A}k_{C})(\sgn B))i$. Using Proposition \ref{ins}(the coefficents A-F in Proposition \ref{ins} depend on $r$, but for convenience of notation we suppress that dependence here) and simplifying, one has $Z_1Z_2=-\frac{rsw}{\tau}+\frac{svr}{\tau}i$.
    \end{proof}
    \subsection{Alternate Proof}
    \begin{proof}Again, we may assume that $Q=Q_{s,t,v,w}$ for some $s,t,v,w$ satisfying (\ref{R1}) and (\ref{R2}). Recall that $\beta=s^2+t^2-2t$. To prove (i), suppose first that one focus, $Z_1$, of $E_0$ equals $EP=a+\frac{1}{2}i$, where $a=\frac{\beta}{2s}$ by (\ref{EP}). Suppose that the other focus is given by $Z-2=c+di$ for some real numbers $c,d$. Then $Z_1Z_2=ac-\frac{1}{2}d+\left(ad+\frac{1}{2}c\right)i$ and by Lemma \ref{fociprod}, $ac-\frac{1}{2}d=\frac{-rsw}{\tau}$ and $\frac{1}{2}c=\frac{rsv}{\tau}$. Let $ \sigma=\frac{2s^2r}{(\beta)(\beta-2t+1)\tau}$. Solving for $c$ and $d$ and substituting for $a$ yields $c=c_1$, where $c_1=-((\beta-t)w-sv) \sigma$ and $d_1=-((\beta-t)v+sw) \sigma$. We also have $Z_1+Z_2=2x_0+2y_0i$, where $(x_0,y_0)$ is the center of $E_0$. Since $Z_1+Z_2=a+c+\left(d+\frac{1}{2}\right)i, a+c=2x_0$ and $d+\frac{1}{2}=2y_0$, which implies that $c=2x_0-a=c_2$, where $c_2=\frac{(\beta-t)(v-s)r+v(s^2-t^2+t)}{2s\tau}$ and $d=2y_0-(1/2)=d_2$ , where
    $d_2=\frac{(s-2vt+v+2ws)r+(2t-1)v}{2\tau}$. Setting $c_2-c_1=0$ and solving for $r$ yields $r=r_{c}$, a rational function in $s,t,v,w$. Similarly, setting $d_2-d_1=0$ and solving for $r$ yields $r=r_{d}$,also a rational function in $s,t,v,w$. We leave out the details. Since both equations $Z_1Z_2=ac-\frac{1}{2}d+\left(ad+\frac{1}{2}c\right)i$ and $Z_1+Z_2=a+c+\left(d+\frac{1}{2}\right)i$ must hold, $r_{c}=r_{d}$. Now $r_{d}-r_{c}$ is a rational function in $s,t,v,w$ whose numerator equals $2sv \beta(\beta-2t+1)^2(a_5v+a_6)$, where $a_5$ and $a_6$ are given in (\ref{a5a6}). Setting $r_{d}-r_{c}=0$ yields $a_5v+a_6=0$. The rest of the proof follows exactly as in the first proof of Theorem \ref{Besant}(i) given above and yields $v=\frac{tI}{s\beta}$ and $w=1-\frac{sv}{t}$(and thus $Q_{s,t,v,w}$ is orthodiagonal). Substituting for $v$ and $w$ into $r_{c}$ and $r_{d}$ gives $r_{c}=r_{d}=\frac{v}{s+v}$, and also using $r=\frac{v}{s+v}$ gives $c_1= c_2=\frac{st}{\beta}$ and $d_1=d_2=\frac{t^2}{\beta}$, which implies that $Z_2$(written without complex notation) equals $IP=\left(\frac{st}{\beta},\frac{t^2}{\beta}\right)$. That proves that if one focus of $E_0$ equals $EP$, then the other focus of $E_0$ equals $IP$. Conversely, suppose first that one focus of $E_0$ equals $IP$ and write that focus using complex notation as $Z_1=a+bi$, where $a=\frac{vs}{N+s}$ and $b=\frac{t}{s}a$. The other focus is given by $Z_2=c+di$ for some real numbers $c,d$. Then $Z_1Z_2=a\left(c-\frac{t}{s}d+(d+\frac{t}{s}c)i\right)$ and by Lemma \ref{fociprod}, $ac-\frac{t}{s}ad=\frac{-rsw}{\tau}$ and $ad+\frac{t}{s}ac=\frac{rsv}{\tau}$. Solving for $c$ and $d$ and substituting for $a$ yields $c=c_1$, where
    $c_1=\frac{sN(N+s)r}{v\beta\tau}$ and $d=d_1$, where $d_1=\frac{s(N+s)(sv+tw)r}{v\beta\tau}$. $Z_1+Z_2=2x_0+2y_0i \Rightarrow a+c=2x_0$ and $\frac{t}{s}a+d=2y_0$, which implies that $c=2x_0-a=c_2$, where $c_2=-sv\frac{(s-v)r+v-N-s)}{(N+s)\tau}$ and
    $d=2y_0-\frac{t}{s}a=d_2$, where $d_2=-\frac{(s^2w^2+t^2v^2-s^2+vts-v^2t-2stvw)r+vt(v-N-s)}{(N+s)\tau}$. As done above, setting
    Setting $c_2-c_1=0$ and solving for $r$ yields $r=r_{c}$, a rational function in $s,t,v,w$. Similarly, setting $d_2-d_1=0$ and solving for $r$ yields $r=r_{d}$,also a rational function in $s,t,v,w$. Again we leave out the details. As before, $r_{d}-r_{c}$ is a rational function in $s,t,v,w$ and this time the numerator equals $(sv+tw-t)(N+s)$. Setting $r_{d}-r_{c}=0$ and solving for $w$ yields $w=1-\frac{sv}{t}$ since $N+s \neq 0$. Thus $Q_{s,t,v,w}$ is orthodiagonal and Lemma \ref{orth} then gives $v=\frac{tI}{s\beta}$. Substituting for $v$ and for $w$ into $r_{c}$ and $r_{d}$ gives $r_{c}=r_{d}=\frac{v}{s+v}$, and also using $r=\frac{v}{s+v}$ gives $c_1= c_2=\frac{\beta-t}{2s}$ and $d_1=d_2=\frac{1}{2}$, which implies that $Z_2$(written without complex notation) equals $EP$. \end{proof}
    \textbf{Appendix}
    \begin{proof}By using an isometry of the plane, we can assume that $Q=Q(A_1,A_2,A_3,A_4)$ has vertices \begin{gather}A_1=(0,0,A_2=(0,u),A_3=(s,t),A_4=(v,w),\nonumber\\\text{where}\hspace{0.15cm} s,v,u>0\hspace{0.15cm}\text{and}\hspace{0.15cm} t \geq w. \label{rest}\end{gather}
    To obtain this isometry, first apply a translation so that one of the vertices of Q, say $A_1$(any vertex will do) equals $(0,0)$. Then rotate the segment $S_{1,2}=\overline{A_1A_2}$, (one could also rotate $\overline{A_1A_4}$) about $(0,0)$ so that $A_2=(0,u)$ for some $u>0$. Then $S_{1,2}$ is vertical and the x coordinates of the other two points are either all positive or all negative. This is true since a convex quadrilateral must lie on one side of any of its sides. If the latter case occurs, then do a reflection thru the y axis. Then the vertices of $Q$ have the form given in (\ref{rest}). Finally, consider the map $T(x,y)=\left(\frac{1}{u}x,\frac{1}{u}y\right)$,$ \neq 0$. By Lemma \ref{scaling}, it suffices to assume that $u=1$ and thus we may also assume that $Q=Q_{s,t,v,w}$ for some $s,t,v,w$ satisfying (\ref{R1}) and (\ref{R2}). Let $\Gamma$ be the circle which passes thru the midpoints of the sides, $MP_1-MP_3$, of $Q_{s,t,v,w}$. It is not hard to show that the equation of $\Gamma$ is given by $f(x,y)=0$, where $f(x,y)=(x-c_1)^2+(y-c_2)^2-c_1^2-\left(\frac{1}{2}-c_2\right)^2,
    c_1=\frac{(w-1)t^2+((w-1)^2+v^2+2sv)t-(w-1)s^2}{4(N+s)}$, and $c_2=\frac{t^2v+2( s+v-sw)t-s(v^2+w^2+sv-1)}{4(N+s)}$. Now $f\left(\frac{v}{2},\frac{w}{2}\right)=\frac{1}{2}(t(1-w)-sv)=0 \iff w=1-\frac{sv}{t} \iff Q_{s,t,v,w}$ is orthodiagonal by Proposition \ref{cycorth}(ii). Substituting $w=1-\frac{sv}{t}$ it follows easily that $c_1= \frac{s+v}{4}$ and $c_2=\frac{t^2+2t-sv}{4t}$, which implies that the center of $\Gamma$ is $C_0=\left(\frac{s+v}{4},\frac{t^2+2t-sv}{4t}\right)$. Now by Proposition \ref{orth}, $v$ is given by (\ref{v}). Substituting for $v$ gives $C_0=\left(\frac{(s^2+t^2)^2+t(s^2-t^2)}{4s(s^2+t^2)},\frac{(s^2+3t^2)}{4(s^2+t^2)}\right) = \frac{1}{2}(EP+IP)$, which completes the proof of Proposition \ref{orthomp}.\end{proof}
    

\begin{thebibliography}{11}
    \bibitem{ref1}
    C. Bond, "A New Algorithm for Scan Conversion of a General Ellipse", preprint, http://www.crbond.com/papers/ell\_alg.pdf
    \bibitem{ref2}
    W.H. Besant, Conic Sections Treated Geometrically, George Bell and Sons Educational Catalogue, Project Gutenberg ebook number 29913, 2009
    \bibitem{ref3}
    W.H. Besant, Solutions of Examples in Conic Sections Treated Geometrically, 3rd edition, revised, Cambridge, 1890
    \bibitem{ref4}
    G. D. Chakerian, A Distorted View of Geometry, MAA, Mathematical Plums, Washington, DC, 1979, 130-150.
    \bibitem{ref5}
    Alan Horwitz, Ellipses Inscribed in, and Circumscribed about, Quadrilaterals, Chapman and Hall, 2024, ISBN 9781032622590, 149 pages.
    \bibitem{ref6}
    Alan Horwitz, "Ellipses of maximal area and of minimal eccentricity inscribed in a convex quadrilateral", Australian Journal of Mathematical Analysis and Applications, Volume 2, Issue 1 (2005), 1-12.
    \bibitem{ref7}
    Alan Horwitz, Dynamics of ellipses inscribed in triangles, Journal of Science, Environment and Technology, Volume 5, Issue 1 (2016), 1-21
    \bibitem{ref8}
    Morris Marden, "A note on the zeros of the sections of a partial fraction, Bulletin of the AMS 51 (1945), 935-940.
    \bibitem{ref9}
    Mohamed Ali Said, "Calibration of an Ellipse's Algebraic Equation and Direct Determination of its Parameters", Acta Mathematica Academiae Paedagogicae Ny regyh aziensis Vol.19, No. 2 (2003), 221-225.
    \bibitem{ref10}
    https://en.m.wikipedia.org/wiki/Conic\_section
    \bibitem{ref11}
    https://mathworld.wolfram.com/Ellipse.html
    \end{thebibliography}
    \end{document}